\documentclass[9pt,twocolumn,twoside]{pnas-new}

\templatetype{pnasresearcharticle} 

\usepackage{svg}
\usepackage{bm}
\usepackage{mathtools}
\usepackage{amsmath}

\usepackage{wasysym}
\usepackage{mathdots}

\algrenewcommand\algorithmicrequire{\textbf{Input:}}
\algrenewcommand\algorithmicensure{\textbf{Output:}}

\algrenewcommand\algorithmicrequire{\textbf{Input:}}
\algrenewcommand\algorithmicensure{\textbf{Output:}}
\usepackage{amsthm}
\usepackage{amsmath}
\newtheorem{theorem}{Theorem}[subsection]

\theoremstyle{remark}
\newtheorem*{remark}{Remark}

\newtheorem{definition}{Definition}[subsection]

\usepackage{mathtools}
\usepackage{amssymb}

\usepackage {mathtools}
\usepackage{amsmath}
\usepackage{algpseudocode}
\usepackage{algorithm}
\usepackage{CJKutf8}
\algnewcommand\algorithmicforeach{\textbf{for each}}
\algdef{S}[FOR]{ForEach}[1]{\algorithmicforeach\ #1\ \algorithmicdo}

\title{Robust estimations from distribution structures: III. Invariant Moments}
\author[a,b,c,1]{Li Tuobang}
\leadauthor{Li}
\significancestatement{Bias, variance, and contamination are the three main errors in statistics. Consistent robust estimation is unattainable without distributional assumptions. In this article, invariant moments are proposed as a means of achieving near-consistent and robust estimations of moments, even in scenarios where moderate violations of distributional assumptions occur, while the variances are sometimes smaller than those of the sample moments.}

\authorcontributions{L.T. designed research, performed research, analyzed data, and wrote the paper.}
\authordeclaration{The author declares no competing interest.}
\correspondingauthor{\textsuperscript{1}To whom correspondence should be addressed. E-mail: tuobang@biomathematics.org}
\begin{abstract}
Descriptive statistics for parametric models are currently highly sensative to departures, gross errors, and/or random errors. Here, leveraging the structures of parametric distributions and their central moment kernel distributions, a class of estimators, consistent simultanously for both a semiparametric distribution and a distinct parametric distribution, is proposed. These efficient estimators are robust to both gross errors and departures from parametric assumptions, making them ideal for estimating the mean and central moments of common unimodal distributions. This article opens up the possibility of utilizing the common nature of probability models to construct near-optimal estimators that are suitable for various scenarios.

\end{abstract}

\dates{This manuscript was compiled on \today}

\begin{document}

\maketitle
\thispagestyle{firststyle}
\ifthenelse{\boolean{shortarticle}}{\ifthenelse{\boolean{singlecolumn}}{\abscontentformatted}{\abscontent}}{}

\vspace*{-10pt}
\dropcap{T}he potential biases of robust location estimators in estimating the population mean have been noticed for more than two centuries \cite{gauss1823theoria}, with numerous significant attempts made to address them. In calculating a robust estimator, the procedure of identifying and downweighting extreme values inherently necessitates the formulation of distributional assumptions. Previously, it was demonstrated that, due to the presence of infinite-dimensional nuisance shape parameters, the semiparametric approach struggles to consistently address distributions with shapes more intricate than symmetry. Newcomb (1886) provided the first modern approach to robust parametric estimation by developing a class of estimators that gives "less weight to the more discordant observations" \cite{newcomb1886generalized}. In 1964, Huber \cite{huber1964robust} used the minimax procedure to obtain $M$-estimator for the contaminated normal distribution, which has played a pre-eminent role in the later development of robust statistics. However, as previously demonstrated, under growing asymmetric departures from normality, the bias of the Huber $M$-estimator (HM) increases rapidly. This is a common issue in parametric robust statistics. For example, He and Fung (1999) constructed \cite{he1999method} a robust $M$-estimator (HFM) for the two-parameter Weibull distribution, from which the mean and central moments can be calculated. Nonetheless, it is inadequate for other parametric distributions, e.g., the gamma, Perato, lognormal, and the generalized Gaussian distributions (SI Dataset S1). Another interesting approach is based on $L$-estimators, such as percentile estimators. For examples of percentile estimators for the Weibull distribution, the reader is referred to the works of Menon (1963) \cite{menon1963estimation}, Dubey (1967) \cite{dubey1967some}, Marks (2005) \cite{marks2005estimation}, and Boudt, Caliskan, and Croux (2011) \cite{boudt2011robust}. At the outset of the study of percentile estimators, it was known that they arithmetically utilize the invariant structures of parametric distributions \cite{menon1963estimation,dubey1967some}. An estimator is classified as an $I$-statistic if it asymptotically satisfies $\text{I}\left(\text{LE}_{1},\ldots,\text{LE}_{l}\right)=\left(\theta_1,\ldots,\theta_q\right)$ for the distribution it is consistent, where $\text{LE}$s are calculated with the use of $LU$-statistics (defined in REDS II), $\text{I}$ is defined using arithmetic operations and constants but may also incorporate transcendental functions and quantile functions, and $\theta$s are the population parameters it estimates. In this article, two subclasses of $I$-statistics are introduced, recombined $I$-statistics and quantile $I$-statistics. Based on $LU$-statistics, $I$-statistics are naturally robust. Compared to probability density functions (pdfs) and cumulative distribution functions (cdfs), the quantile functions of many parametric distributions are more elegant. Since the expectation of an $L$-estimator can be expressed as an integral of the quantile function, $I$-statistics are often analytically obtainable. However, it is observed that even when the sample follows a gamma distribution, which belongs to the same larger family as the Weibull model, the generalized gamma distribution, a misassumption can still lead to substantial biases in Marks percentile estimator (MP) for the Weibull distribution \cite{marks2005estimation} (SI Dataset S1). 

The purpose of this paper is to demonstrate that, in light of previous works, by utilizing the invariant structures of unimodal distributions, a suite of robust estimators can be constructed whose biases are typically smaller than the variances (as seen in Table \ref{tab:comparison} for $n=5184$). 

\subsection{Invariant Moments}\label{II} Most popular robust location estimators, such as the symmetric trimmed mean, symmetric Winsorized mean, Hodges-Lehmann estimator, Huber $M$-estimator, and median of means, are symmetric. As shown in REDS I, a symmetric weighted Hodges-Lehmann mean (${\text{SWHLM}}_{k,\epsilon}$) can achieve consistency for the population mean in any symmetric distribution with a finite mean. However, it falls considerably short of consistently handling other parametric distributions that are not symmetric. Shifting from semiparametrics to parametrics, consider a robust estimator with a non-sample-dependent breakdown point (defined in Subsection \ref{1}) which is consistent simultaneously for both a semiparametric distribution and a parametric distribution that does not belong to that semiparametric distribution, it is named with the prefix ‘invariant’ followed by the name of the population parameter it is consistent with. Here, the recombined $I$-statistic is defined as \begin{align*}\text{RI}_{d,h_{\mathbf{k}},\mathbf{k}_1,\mathbf{k}_2,k_1,k_2,\epsilon=\min{(\epsilon_1,\epsilon_2)},\gamma_1,\gamma_2,n,LU_1,LU_2}\coloneqq\\\lim_{c\to\infty}{\left(\frac{\left({LU_1}_{h_{\mathbf{k}},\mathbf{k}_1,k_1,\epsilon_1,\gamma_1,n}+c\right)^{d+1}}{\left({LU_2}_{h_{\mathbf{k}},\mathbf{k}_2,k_2,\epsilon_2,\gamma_2,n}+c\right)^d}-c\right)}\text{,}\end{align*} where $d$ is the key factor for bias correction, $LU_{h_{\mathbf{k}},\mathbf{k},k,\epsilon, \gamma,n}$ is the $LU$-statistic, $\mathbf{k}$ is the degree of the $U$-statistic, $k$ is the degree of the $LL$-statistic, $\epsilon$ is the upper asymptotic breakdown point of the $LU$-statistic. It is assumed in this series that in the subscript of an estimator, if $\mathbf{k}$, $k$ and $\gamma$ are omitted, $\mathbf{k}=1$, $k=1$, $\gamma=1$ are assumed, if just one $\mathbf{k}$ is indicated, $\mathbf{k}_1=\mathbf{k}_2$, if just one $\gamma$ is indicated, $\gamma_1=\gamma_2$, if $n$ is omitted, only the asymptotic behavior is considered, in the absence of subscripts, no assumptions are made. The subsequent theorem shows the significance of a recombined $I$-statistic.

\begin{theorem}\label{rm} Define the recombined mean as ${rm}_{d,k_1,k_2,\epsilon=\min{(\epsilon_1,\epsilon_2)},\gamma_1,\gamma_2,n,\text{WL}_1,\text{WL}_2}\coloneqq\text{RI}_{d,h_{\mathbf{k}}=x,\mathbf{k}_1=1,\mathbf{k}_2=1,k_1,k_2,\epsilon=\min{(\epsilon_1,\epsilon_2)},\gamma_1,\gamma_2,n,LU_1=\text{WL}_1,LU_2=\text{WL}_2}$. Assuming finite means, ${rm}_{d=\frac{\mu-{\text{WL}_1}_{k_1,\epsilon_1,\gamma_1}}{{\text{WL}_1}_{k_1,\epsilon_1,\gamma_1}-{\text{WL}_2}_{k_2,\epsilon_2,\gamma_2}},k_1,k_2,\epsilon=\min{(\epsilon_1,\epsilon_2)},\gamma_1,\gamma_2,\text{WL}_1,\text{WL}_2}$ is a consistent mean estimator for a location-scale distribution, where $\mu$, ${\text{WL}_1}_{k_1,\epsilon_1,\gamma_1}$, and ${\text{WL}_2}_{k_2,\epsilon_2,\gamma_2}$ are different location parameters from that location-scale distribution. If $\gamma_1=\gamma_2=1$, $\text{WL}=\text{SWHLM}$, ${rm}$ is also consistent for any symmetric distributions.\end{theorem}

\begin{proof} Finding $d$ that make ${rm}_{d,k_1,k_2,\epsilon=\min{(\epsilon_1,\epsilon_2)},\gamma_1,\gamma_2,\text{WL}_1,\text{WL}_2}$ a consistent mean estimator is equivalent to finding the solution of ${rm}_{d,k_1,k_2,\epsilon=\min{(\epsilon_1,\epsilon_2)},\gamma_1,\gamma_2,\text{WL}_1,\text{WL}_2}=\mu$. First consider the location-scale distribution. Since ${rm}_{d,k_1,k_2,\epsilon=\min{(\epsilon_1,\epsilon_2)},\gamma_1,\gamma_2,\text{WL}_1,\text{WL}_2}=\lim_{c\to\infty}{\left(\frac{\left({\text{WL}_1}_{k_1,\epsilon_1,\gamma_1}+c\right)^{d+1}}{\left({\text{WL}_2}_{k_2,\epsilon_2,\gamma_2}+c\right)^d}-c\right)}=\left(d+1\right){\text{WL}_1}_{k_1,\epsilon_1,\gamma}-d {\text{WL}_2}_{k_2,\epsilon_2,\gamma}=\mu$. So, $d=\frac{\mu-{\text{WL}_1}_{k_1,\epsilon_1,\gamma_1}}{{\text{WL}_1}_{k_1,\epsilon_1,\gamma_1}-{\text{WL}_2}_{k_2,\epsilon_2,\gamma_2}}$. In REDS I, it was established that any $\mathrm{WL}(k,\epsilon,\gamma)$ can be expressed as $\lambda \mathrm{WL}_{0}(k,\epsilon,\gamma)+\mu$ for a location-scale distribution parameterized by a location parameter $\mu$ and a scale parameter $\lambda$, where $\mathrm{WL}_{0}(k,\epsilon,\gamma)$ is a function of $Q_0(p)$, the quantile function of a standard distribution without any shifts or scaling, according to the definition of the weighted $L$-statistic. The simultaneous cancellation of $\mu$ and $\lambda$ in $\frac{(\lambda \mu_0+\mu)-(\lambda {\text{WL}_1}_{0}(k_1,\epsilon_1,\gamma_1)+\mu)}{(\lambda {\text{WL}_1}_{0}(k_1,\epsilon_1,\gamma_1)+\mu)-(\lambda {\text{WL}_2}_{0}(k_2,\epsilon_2,\gamma_2)+\mu)}$ assures that the $d$ in $rm$ is always a constant for a location-scale distribution. The proof of the second assertion follows directly from the coincidence property. According to Theorem 17 in REDS I, for any symmetric distribution with a finite mean, ${\text{SWHLM}_1}_{k_1}={\text{SWHLM}_2}_{k_2}=\mu$. Then ${rm}_{d,k_1,k_2,\epsilon_1,\epsilon_2,\text{SWHLM}_1,\text{SWHLM}_2}=\lim_{c\to\infty}{\left(\frac{\left(\mu+c\right)^{d+1}}{\left(\mu+c\right)^d}-c\right)}=\mu\text{.} $ This completes the demonstration.\end{proof}

For example, the Pareto distribution has a quantile function $Q_{Par}\left(p\right)=x_m(1-p)^{-\frac{1}{\alpha}}$, where $x_m$ is the minimum possible value that a random variable following the Pareto distribution can take, serving a scale parameter, $\alpha$ is a shape parameter. The mean of the Pareto distribution is given by $\frac{\alpha x_m}{\alpha-1}$. As $\mathrm{WL}(k,\epsilon,\gamma)$ can be expressed as a function of $Q(p)$, one can set the two $\text{WL}_{k,\epsilon,\gamma}$s in the $d$ value of $rm$ as two arbitrary quantiles $Q_{Par}(p_1)$ and $Q_{Par}(p_2)$. For the Pareto distribution, $d_{Per,rm}=\frac{\mu_{Per}-Q_{Par}(p_1)}{Q_{Par}(p_1)-Q_{Par}(p_2)}=\frac{\frac{\alpha x_m}{\alpha-1}-x_m(1-p_1)^{-\frac{1}{\alpha}}}{x_m(1-p_1)^{-\frac{1}{\alpha}}-x_m(1-p_2)^{-\frac{1}{\alpha}}}$. $x_m$ can be canceled out. Intriguingly, the quantile function of exponential distribution is $Q_{exp}(p)=\ln \left(\frac{1}{1-p}\right)\lambda$, $\lambda\geq0$. $\mu_{exp}=\lambda$. Then, $d_{exp,rm}=\frac{\mu_{exp}-Q_{exp}(p_1)}{Q_{exp}(p_1)-Q_{exp}(p_2)}=\frac{\lambda-\ln \left(\frac{1}{1-p_1}\right)\lambda}{\ln \left(\frac{1}{1-p_1}\right)\lambda-\ln \left(\frac{1}{1-p_2}\right)\lambda}=-\frac{\ln (1-p_1)+1}{\ln (1-p_1)-\ln (1-p_2)}$. Since $\lim_{\alpha\to\infty}{\frac{\frac{\alpha }{\alpha-1}-(1-p_1)^{-1/\alpha}}{(1-p_1)^{-1/\alpha}-(1-p_2)^{-1/\alpha}}}=-\frac{\ln (1-p_1)+1}{\ln (1-p_1)-\ln (1-p_2)}$, $d_{Per,rm}$ approaches $d_{exp,rm}$, as $\alpha\to\infty$, regardless of the type of weighted $L$-statistic used. That means, for the Weibull, gamma, Pareto, lognormal and generalized Gaussian distribution, ${rm}_{d=\frac{\mu-{\text{SWHLM}_1}_{k_1,\epsilon_1}}{{\text{SWHLM}_1}_{k_1,\epsilon_1}-{\text{SWHLM}_2}_{k_2,\epsilon_2}},k_1,k_2,\epsilon=\min{(\epsilon_1,\epsilon_2)},\text{SWHLM}_1,\text{SWHLM}_2}$ is consistent for at least one particular case, where $\mu$, ${\text{SWHLM}_1}_{k_1,\epsilon_1}$, and ${\text{SWHLM}_2}_{k_2,\epsilon_2}$ are different location parameters from an exponential distribution. Let ${\text{SWHLM}_1}_{k_1,\epsilon_1,\gamma}=\text{BM}_{\nu=3,\epsilon=\frac{1}{24}}$, ${\text{SWHLM}_2}_{k_2,\epsilon_2,\gamma}=m$, then $\mu=\lambda$, $m=Q\left(\frac{1}{2}\right)=\ln{2}\lambda$, $\text{BM}_{\nu=3,\epsilon=\frac{1}{24}}=\lambda \left(1+\ln \left(\frac{26068394603446272 \sqrt[6]{\frac{7}{247}} \sqrt[3]{11}}{391^{5/6}101898752449325 \sqrt{5} }\right)\right)$, the detailed formula is given in the SI Text. So, $d=\frac{\mu-{\mathrm{BM}}_{\nu=3,\epsilon=\frac{1}{24}}}{{\mathrm{BM}}_{\nu=3,\epsilon=\frac{1}{24}}-m}=\frac{\lambda-\lambda \left(1+\ln \left(\frac{26068394603446272 \sqrt[6]{\frac{7}{247}} \sqrt[3]{11}}{391^{5/6}101898752449325 \sqrt{5} }\right)\right)}{\lambda \left(1+\ln \left(\frac{26068394603446272 \sqrt[6]{\frac{7}{247}} \sqrt[3]{11}}{391^{5/6}101898752449325 \sqrt{5} }\right)\right)-\ln{2}\lambda}=-\frac{\ln \left(\frac{26068394603446272 \sqrt[6]{\frac{7}{247}} \sqrt[3]{11}}{391^{5/6}101898752449325 \sqrt{5} }\right)}{1-\ln (2)+\ln \left(\frac{26068394603446272 \sqrt[6]{\frac{7}{247}} \sqrt[3]{11}}{391^{5/6}101898752449325 \sqrt{5} }\right)}\approx0.103$. The biases of ${rm}_{d\approx0.103,\nu=3,\epsilon=\frac{1}{24},\text{BM},m}$ for distributions with skewness between those of the exponential and symmetric distributions are tiny (SI Dataset S1). ${rm}_{d\approx0.103,\nu=3,\epsilon=\frac{1}{24},\text{BM},m}$ exhibits excellent performance for all these common unimodal distributions (SI Dataset S1).

The recombined mean is a recombined $I$-statistic. Consider an $I$-statistic whose $\text{LE}$s are percentiles of a distribution obtained by plugging $LU$-statistics into a cumulative distribution function, $\text{I}$ is defined with arithmetic operations, constants, and quantile functions, such an estimator is classified as a quantile $I$-statistic. One version of the quantile $I$-statistic can be defined as $\text{QI}_{d,h_{\mathbf{k}},\mathbf{k},k,\epsilon,\gamma,,n,LU}\coloneqq\\ \begin{cases} \hat{Q}_{n,h_{\mathbf{k}}}\left(\left(\hat{F}_{n,h_{\mathbf{k}}}\left(LU\right)-\frac{\gamma}{1+\gamma}\right)d+\hat{F}_{n,h_{\mathbf{k}}}\left(LU\right)\right) & \hat{F}_{n,h_{\mathbf{k}}}\left(LU\right)\geq\frac{\gamma}{1+\gamma} \\ \hat{Q}_{n,h_{\mathbf{k}}}\left(\hat{F}_{n,h_{\mathbf{k}}}\left(LU\right)-\left(\frac{\gamma}{1+\gamma}-\hat{F}_{n,h_{\mathbf{k}}}\left(LU\right)\right)d\right) & \hat{F}_{n,h_{\mathbf{k}}}\left(LU\right)<\frac{\gamma}{1+\gamma}, \end{cases}$ where $LU$ is ${LU}_{\mathbf{k},k,\epsilon,\gamma,n}$, $\hat{F}_{n,h_{\mathbf{k}}}\left(x\right)$ is the empirical cumulative distribution function of the $h_{\mathbf{k}}$ kernel distribution, $\hat{Q}_{n,h_{\mathbf{k}}}$ is the quantile function of the $h_{\mathbf{k}}$ kernel distribution.

Similarly, the quantile mean can be defined as $qm_{d,k,\epsilon,\gamma,n,\mathrm{WL}}\coloneqq\text{QI}_{d,h_{\mathbf{k}}=x,\mathbf{k}=1,k,\epsilon,\gamma,n,LU=\mathrm{WL}}$. Moreover, in extreme right-skewed heavy-tailed distributions, if the calculated percentile exceeds $1-\epsilon$, it will be adjusted to $1-\epsilon$. In a left-skewed distribution, if the obtained percentile is smaller than $\gamma\epsilon$, it will also be adjusted to $\gamma\epsilon$. Without loss of generality, in the following discussion, only the case where $\hat{F}_{n}\left(\text{WL}_{k,\epsilon,\gamma,n}\right)\geq\frac{\gamma}{1+\gamma}$ is considered. The most popular method for computing the sample quantile function was proposed by Hyndman and Fan in 1996 \cite{hyndman1996sample}. Another widely used method for calculating the sample quantile function involves employing linear interpolation of modes corresponding to the order statistics of the uniform distribution on the interval [0, 1], i.e., $\hat{Q}_{n}\left(p\right)=X_{\left\lfloor h\right\rfloor}+\left(h-\left\lfloor h\right\rfloor\right)\left(X_{\left\lceil h\right\rceil}-X_{\left\lfloor h\right\rfloor}\right)$, $h=\left(n-1\right)p+1$. To minimize the finite sample bias, here, the inverse function of $\hat{Q}_{n}$ is deduced as $\hat{F}_{n}\left(x\right)\coloneqq \frac{1}{n}\left(\frac{x-X_{cf}}{X_{cf+1}-X_{cf}}+cf\right)$, based on Hyndman and Fan's definition, or $\hat{F}_{n}\left(x\right)\coloneqq \frac{1}{n-1} \left(cf - 1 + \frac{x - X_{cf}}{X_{cf+1} - X_{cf}}\right)$, based on the latter definition, where $cf = \sum_{i=1}^{n} \mathbf{1}_{X_i \le x}$, $\mathbf{1}_A$ is the indicator of event $A$.

The quantile mean uses the location-scale invariant in a different way, as shown in the subsequent proof.

\begin{theorem}\label{qm}${qm}_{d=\frac{F(\mu)-F({\text{WL}}_{k,\epsilon,\gamma})}{F({\text{WL}}_{k,\epsilon,\gamma})-\frac{\gamma}{1+\gamma}},k,\epsilon,\gamma,\mathrm{WL}}$ is a consistent mean estimator for a location-scale distribution provided that the means are finite and $F(\mu)$, $F({\text{WL}}_{k,\epsilon,\gamma})$ and $\frac{\gamma}{1+\gamma}$ are all within the range of $[\gamma\epsilon,1-\epsilon]$, where $\mu$ and ${\text{WL}}_{k,\epsilon,\gamma}$ are location parameters from that location-scale distribution. If $\text{WL}=\text{SWHLM}$, ${qm}$ is also consistent for any symmetric distributions. \end{theorem}\begin{proof}When $F\left(\text{WL}_{k,\epsilon,\gamma}\right)\geq\frac{\gamma}{1+\gamma}$, the solution of $\left(F\left(\text{WL}_{k,\epsilon,\gamma}\right)-\frac{\gamma}{1+\gamma}\right)d+F\left(\text{WL}_{k,\epsilon,\gamma}\right)=F\left(\mu\right)$ is $d=\frac{F(\mu)-F(\text{WL}_{k,\epsilon,\gamma})}{F(\text{WL}_{k,\epsilon,\gamma})-\frac{\gamma}{1+\gamma}}$. The $d$ value for the case where $F\left(\text{WL}_{k,\epsilon,\gamma,n}\right)<\frac{\gamma}{1+\gamma}$ is the same. The definitions of the location and scale parameters are such that they must satisfy $F(x;\lambda,\mu)=F(\frac{x-\mu}{\lambda};1,0)$, then $F(\mathrm{WL}(k,\epsilon,\gamma);\lambda,\mu)=F(\frac{\lambda \mathrm{WL}_{0}(k,\epsilon,\gamma)+\mu-\mu}{\lambda};1,0)=F(\mathrm{WL}_{0}(k,\epsilon,\gamma);1,0)$. It follows that the percentile of any weighted $L$-statistic is free of $\lambda$ and $\mu$ for a location-scale distribution. Therefore $d$ in $qm$ is also invariably a constant. For the symmetric case, $F\left(\text{SWHLM}_{k,\epsilon}\right)=F\left(\mu\right)=F(Q(\frac{1}{2}))=\frac{1}{2}$ is valid for any symmetric distribution with a finite second moment, as the same values correspond to same percentiles. Then, $qm_{d,k,\epsilon,\text{SWHLM}}=F^{-1}\left(\left(F\left(\text{SWHLM}_{k,\epsilon}\right)-\frac{1}{2}\right)d+F\left(\mu\right)\right)=F^{-1}\left(0+F\left(\mu\right)\right)=\mu.$ To avoid inconsistency due to post-adjustment, $F(\mu)$, $F({\text{WL}}_{k,\epsilon,\gamma})$ and $\frac{\gamma}{1+\gamma}$ must reside within the range of $[\gamma\epsilon,1-\epsilon]$. All results are now proven.\end{proof} 

The cdf of the Pareto distribution is $F_{Par}(x)=1-\left(\frac{x_m}{x}\right)^{\alpha}$. So, set the $d$ value in $qm$ with two arbitrary percentiles $p_1$ and $p_2$, $d_{Par,qm}=\frac{1-\left(\frac{x_m}{\frac{\alpha x_m}{\alpha-1}}\right)^\alpha-\left(1-\left(\frac{x_m}{x_m\left(1-p_1\right)^{-\frac{1}{\alpha}}}\right)^\alpha\right)}{\left(1-\left(\frac{x_m}{x_m\left(1-p_1\right)^{-\frac{1}{\alpha}}}\right)^\alpha\right)-\left(1-\left(\frac{x_m}{x_m\left(1-p_2\right)^{-\frac{1}{\alpha}}}\right)^\alpha\right)}=\frac{1-\left(\frac{\alpha-1}{\alpha}\right)^\alpha-p_1}{p_1-p_2}$. The $d$ value in $qm$ for the exponential distribution is always identical to $d_{Par,qm}$ as $\alpha\to\infty$, since $\lim_{\alpha\to\infty}{\left(\frac{\alpha-1}{\alpha}\right)^\alpha}=\frac{1}{e}$ and the cdf of the exponential distribution is $F_{exp}\left(x\right)=1-e^{-\lambda^{-1} x}$, then $d_{exp,qm}=\frac{\left(1-e^{-1}\right)-\left(1-e^{-\ln{\left(\frac{1}{1-p_1}\right)}\ }\right)}{\left(1-e^{-\ln{\left(\frac{1}{1-p_1}\right)}\ }\right)-\left(1-e^{-\ln{\left(\frac{1}{1-p_2}\right)}\ }\right)}=\frac{1-\frac{1}{e}-p_1}{p_1-p_2}$. So, for the Weibull, gamma, Pareto, lognormal and generalized Gaussian distribution, ${qm}_{d=\frac{F_{exp}(\mu)-F_{exp}({\text{SWHLM}}_{k,\epsilon})}{F_{exp}({\text{SWHLM}}_{k,\epsilon})-\frac{1}{2}},k,\epsilon,\text{SWHLM}}$ is also consistent for at least one particular case, provided that $\mu$ and ${\text{SWHLM}}_{k,\epsilon}$ are different location parameters from an exponential distribution and $F(\mu)$, $F({\text{SWHLM}}_{k,\epsilon})$ and $\frac{1}{2}$ are all within the range of $[\epsilon,1-\epsilon]$. Also let ${\text{SWHLM}}_{k,\epsilon,\gamma}=\text{BM}_{\nu=3,\epsilon=\frac{1}{24}}$ and $\mu=\lambda$, then $d=\frac{F_{exp}(\mu)-F_{exp}(\text{BM}_{\nu=3,\epsilon=\frac{1}{24}})}{F_{exp}(\text{BM}_{\nu=3,\epsilon=\frac{1}{24}})-\frac{1}{2}}=\frac{-e^{-1}+e^{-\left(1+\ln \left(\frac{26068394603446272 \sqrt[6]{\frac{7}{247}} \sqrt[3]{11}}{391^{5/6}101898752449325 \sqrt{5} }\right)\right)}}{\frac{1}{2}-e^{-\left(1+\ln \left(\frac{26068394603446272 \sqrt[6]{\frac{7}{247}} \sqrt[3]{11}}{391^{5/6}101898752449325 \sqrt{5} }\right)\right)}}=\frac{\frac{101898752449325 \sqrt{5} \sqrt[6]{\frac{247}{7}} 391^{5/6}}{26068394603446272 \sqrt[3]{11} e}-\frac{1}{e}}{\frac{1}{2}-\frac{101898752449325 \sqrt{5} \sqrt[6]{\frac{247}{7}} 391^{5/6}}{26068394603446272 \sqrt[3]{11} e}}\approx0.088.$ $F_{exp}(\mu)$, $F_{exp}(\text{BM}_{\nu=3,\epsilon=\frac{1}{24}})$ and $\frac{1}{2}$ are all within the range of $[\frac{1}{24},\frac{23}{24}]$. ${qm}_{d\approx0.088,\nu=3,\epsilon=\frac{1}{24},\text{BM}}$ works better in the fat-tail scenarios (SI Dataset S1). Theorem \ref{rm} and \ref{qm} show that ${rm}_{d\approx0.103,\nu=3,\epsilon=\frac{1}{24},\text{BM},m}$ and ${qm}_{d\approx0.088,\nu=3,\epsilon=\frac{1}{24},\text{BM}}$ are both consistent mean estimators for any symmetric distribution and the exponential distribution with finite second moments. It’s obvious that the asymptotic breakdown points of ${rm}_{d\approx0.103,\nu=3,\epsilon=\frac{1}{24},\text{BM},m}$ and ${qm}_{d\approx0.088,\nu=3,\epsilon=\frac{1}{24},\text{BM}}$ are both $\frac{1}{24}$. Therefore they are all invariant means.

To study the impact of the choice of WLs in $rm$ and $qm$, it is constructive to recall that a weighted $L$-statistic is a combination of order statistics. While using a less-biased weighted $L$-statistic can generally enhance performance (SI Dataset S1), there is a greater risk of violation in the semiparametric framework. However, the mean-$\text{WA}_{\epsilon,\gamma}$-$\gamma$-median inequality is robust to slight fluctuations of the QA function of the underlying distribution. Suppose for a right-skewed distribution, the QA function is generally decreasing with respect to $\epsilon$ in $[0,u]$, but increasing in $[u,\frac{1}{1+\gamma}]$, since all quantile averages with breakdown points from $\epsilon$ to $\frac{1}{1+\gamma}$ will be included in the computation of $\text{WA}_{\epsilon,\gamma}$, as long as $\frac{1}{1+\gamma}-u\ll \frac{1}{1+\gamma}-\gamma\epsilon$, and other portions of the QA function satisfy the inequality constraints that define the $\nu$th $\gamma$-orderliness on which the $\text{WA}_{\epsilon,\gamma}$ is based, if $0\leq\gamma\leq1$, the mean-$\text{WA}_{\epsilon,\gamma}$-$\gamma$-median inequality still holds. This is due to the violation of $\nu$th $\gamma$-orderliness being bounded, when $0\leq\gamma\leq1$, as shown in REDS I and therefore cannot be extreme for unimodal distributions with finite second moments. For instance, the SQA function of the Weibull distribution is non-monotonic with respect to $\epsilon$ when the shape parameter $\alpha>\frac{1}{1-\ln (2)}\approx3.259$ as shown in the SI Text of REDS I, the violation of the second and third orderliness starts near this parameter as well, yet the mean-$\text{BM}_{\nu=3,\epsilon=\frac{1}{24}}$-median inequality retains valid when $\alpha\leq3.387$. Another key factor in determining the risk of violation of orderliness is the skewness of the distribution. In REDS I, it was demonstrated that in a family of distributions differing by a skewness-increasing transformation in van Zwet's sense, the violation of orderliness, if it happens, only occurs as the distribution nears symmetry \cite{van1964convex}. When $\gamma=1$, the over-corrections in $rm$ and $qm$ are dependent on the $\text{SWA}_{\epsilon}$-median difference, which can be a reasonable measure of skewness after standardization \cite{bowley1926elements,groeneveld1984measuring}, implying that the over-correction is often tiny with moderate $d$. This qualitative analysis suggests the general reliability of $rm$ and $qm$ based on the mean-$\text{WA}_{\epsilon,\gamma}$-$\gamma$-median inequality, especially for unimodal distributions with finite second moments when $0\leq\gamma\leq1$. Extending this rationale to other weighted $L$-statistics is possible, since the $\gamma$-$U$-orderliness can also be bounded with certain assumptions, as discussed previously.

Consider two continuous distributions belonging to the same location–scale family, according to Theorem 3 in REDS II, their corresponding $\mathbf{k}$th central moment kernel distributions only differ in scaling. Although strict complete $\nu$th orderliness is difficult to prove, following the same logic as discussed above, even if the inequality may be violated in a small range, the mean-$\text{SWA}_{\epsilon}$-median inequality remains valid, in most cases, for the central moment kernel distribution. Define the recombined $\mathbf{k}$th central moment as ${r\mathbf{k}m}_{d,k_1,k_2,\epsilon=\min{(\epsilon_1,\epsilon_2)},\gamma_1,\gamma_2,n,\text{WHL}\mathbf{k}m_1,\text{WHL}\mathbf{k}m_2}\coloneqq\text{RI}_{d,h_{\mathbf{k}}=\psi_\mathbf{k},\mathbf{k}_1=\mathbf{k},\mathbf{k}_2=\mathbf{k},k_1,k_2,\epsilon_1,\epsilon_2,\gamma_1,\gamma_2,n,LU_1=\text{WHL}\mathbf{k}m_1,LU_2=\text{WHL}\mathbf{k}m_2}$. Then, assuming finite $\mathbf{k}$th central moment and applying the same logic as in Theorem \ref{rm}, ${r\mathbf{k}m}_{d=\frac{\mu_\mathbf{k}-{\text{WHL}\mathbf{k}m_1}_{k_1,\epsilon_1,\gamma_1}}{{\text{WHL}\mathbf{k}m_1}_{k_1,\epsilon_1,\gamma_1}-{\text{WHL}\mathbf{k}m_2}_{k_2,\epsilon_2,\gamma_2}},k_1,k_2,\epsilon=\min{(\epsilon_1,\epsilon_2)},\gamma_1,\gamma_2,\text{WHL}\mathbf{k}m_1,\text{WHL}\mathbf{k}m_2}$ is a consistent $\mathbf{k}$th central moment estimator for a location-scale distribution, where $\mu_\mathbf{k}$, ${\text{WHL}\mathbf{k}m_1}_{k_1,\epsilon_1,\gamma_1}$, and ${\text{WHL}\mathbf{k}m_2}_{k_2,\epsilon_2,\gamma_2}$ are different $\mathbf{k}$th central moment parameters from that location-scale distribution. Similarly, the quantile will not change after scaling. The quantile $\mathbf{k}$th central moment is thus defined as \begin{align*}{q\mathbf{k}m}_{d,k,\epsilon,\gamma,n,\text{WHL}\mathbf{k}m}\coloneqq\text{QI}_{d,h_{\mathbf{k}}=\psi_\mathbf{k},\mathbf{k}=\mathbf{k},k,\epsilon,\gamma,n,LU=\text{WHL}\mathbf{k}m}  \text{.}\end{align*} ${q\mathbf{k}m}_{d=\frac{F_{\psi_{\mathbf{k}}}(\mu_\mathbf{k})-F_{\psi_{\mathbf{k}}}({\text{WHL}\mathbf{k}m}_{k,\epsilon,\gamma})}{F_{\psi_{\mathbf{k}}}({\text{WHL}\mathbf{k}m}_{k,\epsilon,\gamma})-\frac{\gamma}{1+\gamma}},k,\epsilon,\gamma,\text{WHL}\mathbf{k}m}$ is also a consistent $\mathbf{k}$th central moment estimator for a location-scale distribution provided that the $\mathbf{k}$th central moment is finite and $F_{\psi_{\mathbf{k}}}(\mu_\mathbf{k})$, $F_{\psi_{\mathbf{k}}}({\text{WHL}\mathbf{k}m}_{k,\epsilon,\gamma})$ and $\frac{\gamma}{1+\gamma}$ are all within the range of $[\gamma\epsilon,1-\epsilon]$, where $\mu_\mathbf{k}$ and ${\text{WHL}\mathbf{k}m}_{k,\epsilon,\gamma}$ are different $\mathbf{k}$th central moment parameters from that location-scale distribution. According to Theorem 2 in REDS II, if the original distribution is unimodal, the central moment kernel distribution is always a heavy-tailed distribution, as the degree term amplifies its skewness and tailedness. From the better performance of the quantile mean in heavy-tailed distributions, the quantile $\mathbf{k}$th central moments are generally better than the recombined $\mathbf{k}$th central moments regarding asymptotic bias.

Finally, the recombined standardized $\mathbf{k}$th moment is defined to be \begin{align*}rs\mathbf{k}m_{\epsilon=\min{(\epsilon_1,\epsilon_2)},k_1,k_2,k_3,k_4,\gamma_1,\gamma_2,\gamma_3,\gamma_4,n,\text{WHL}\mathbf{k}m_1,\text{WHL}\mathbf{k}m_2,\text{WHL}var_1,\text{WHL}var_2}\coloneqq\\ \frac{{r\mathbf{k}m}_{d,k_1,k_2,\epsilon_1,\gamma_1,\gamma_2,n,\text{WHL}\mathbf{k}m_1,\text{WHL}\mathbf{k}m_2}}{({rvar}_{d,k_3,k_4,\epsilon_2,\gamma_3,\gamma_4,n,\text{WHL}var_1,\text{WHL}var_2})^{\mathbf{k}/2}}\text{.}\end{align*} The quantile standardized $\mathbf{k}$th moment is defined similarly, \begin{align*}qs\mathbf{k}m_{\epsilon=\min{(\epsilon_1,\epsilon_2)},k_1,k_2,\gamma_1,\gamma_2,n,\text{WHL}\mathbf{k}m,\text{WHL}var}\coloneqq \frac{{q\mathbf{k}m}_{d,k_1,\epsilon_1,\gamma_1,n,\text{WHL}\mathbf{k}m}}{({qvar}_{d,k_2,\epsilon_2,\gamma_2,n,\text{WHL}var})^{\mathbf{k}/2}}\text{.}\end{align*} 
\subsection{A shape-scale distribution as the consistent distribution}\label{sectionE}In the last section, the parametric robust estimation is limited to a location-scale distribution, with the location parameter often being omitted for simplicity. For improved fit to observed skewness or kurtosis, shape-scale distributions with shape parameter ($\alpha$) and scale parameter ($\lambda$) are commonly utilized. Weibull, gamma, Pareto, lognormal, and generalized Gaussian distributions (when $\mu$ is a constant) are all shape-scale unimodal distributions. Furthermore, if either the shape parameter $\alpha$ or the skewness or kurtosis is constant, the shape-scale distribution is reduced to a location-scale distribution. Let $D(|skewness|,kurtosis,\mathbf{k},etype,dtype,n)=d_{i\mathbf{k}m}$ denote the function to specify $d$ values, where the first input is the absolute value of the skewness, the second input is the kurtosis, the third is the order of the central moment (if $\mathbf{k}=1$, the mean), the fourth is the type of estimator, the fifth is the type of consistent distribution, and the sixth input is the sample size. For simplicity, the last three inputs will be omitted in the following discussion. Hold in awareness that since skewness and kurtosis are interrelated, specifying $d$ values for a shape-scale distribution only requires either skewness or kurtosis, while the other may be also omitted. Since many common shape-scale distributions are always right-skewed (if not, only the right-skewed or left-skewed part is used for calibration, while the other part is omitted), the absolute value of the skewness should be the same as the skewness of these distributions. This setting also handles the left-skew scenario well.

For recombined moments up to the fourth ordinal, the object of using a shape-scale distribution as the consistent distribution is to find solutions for the system of equations $\begin{cases} rm\left(\text{WHLM},\gamma m,D(|rskew|,rkurt,1)\right)=\mu \\ rvar\left(\text{WHL}var,\gamma mvar,D(|rskew|,rkurt,2)\right)=\mu_2 \\ rtm\left(\text{WHL}tm,\gamma mtm,D(|rskew|,rkurt,3)\right)=\mu_3 \\ rfm\left(\text{WHL}fm,\gamma mfm,D(|rskew|,rkurt,4\right)=\mu_4 \\ rskew=\frac{\mu_3}{\mu_2^{\frac{3}{2}}} \\rkurt=\frac{\mu_4}{\mu_2^2} \end{cases}$, where $\mu_2$, $\mu_3$ and $\mu_4$ are the population second, third and fourth central moments. $|rskew|$ and $rkurt$ should be the invariant points of the functions $\varsigma(|rskew|)=\left|\frac{rtm\left(\text{WHL}tm,\gamma mtm,D(|rskew|, 3)\right)}{rvar\left(\text{WHL}var,\gamma mvar,D(|rskew|, 2)\right)^{\frac{3}{2}}}\right |$ and $\varkappa(rkurt)=\frac{rfm\left(\text{WHL}fm,\gamma mfm,D(rkurt, 4)\right)}{rvar\left(\text{WHL}var,\gamma mvar,D(rkurt,2)\right)^2}$. Clearly, this is an overdetermined nonlinear system of equations, given that the skewness and kurtosis are interrelated for a shape-scale distribution. Since an overdetermined system constructed with random coefficients is almost always inconsistent, it is natural to optimize them separately using the fixed-point iteration (see Algorithm \ref{alg:algorithm1}, only $rkurt$ is provided, others are the same).

\begin{algorithm}[H]    
    \caption{$rkurt$ for a shape-scale distribution}\label{alg:algorithm1}
    \begin{algorithmic}[2]
        \Require{$D$; $\text{WHL}var$; $\text{WHL}fm$; $\gamma mvar$; $\gamma mfm$; $maxit$; $\delta$}
        \Ensure{$rkurt_{i-1}$}
        \State $i=0$
        \State $rkurt_{i} \leftarrow$ $\varkappa(kurtosis_{max})$ \Comment{Using the maximum kurtosis available in $D$ as an initial guess.}
        \Repeat
        \State $i=i+1$
        \State $rkurt_{i-1} \leftarrow$ $rkurt_{i}$
        \State $rkurt_{i} \leftarrow$ $\varkappa(rkurt_{i-1})$
        \Until $i>maxit$ or {$|rkurt_{i}-rkurt_{i-1}|<\delta$} \Comment{$maxit$ is the maximum number of iterations, $\delta$ is a small positive number.}
    \end{algorithmic}
\end{algorithm}

The following theorem shows the validity of Algorithm \ref{alg:algorithm1}. \begin{theorem}\label{twoc} Assuming $\gamma=1$ and $m\mathbf{k}m$s, where $2\leq\mathbf{k}\leq4$, are all equal to zero, $|rskew|$ and $rkurt$, defined as the largest attracting fixed points of the functions $\varsigma(|rskew|)$ and $\varkappa(rkurt)$, are consistent estimators of $\Tilde{\mu}_{3}$ and $\Tilde{\mu}_{4}$ for a shape-scale distribution whose $\mathbf{k}$th central moment kernel distributions are $U$-congruent, as long as they are within the domain of $D$, where $\Tilde{\mu}_{3}$ and $\Tilde{\mu}_{4}$ are the population skewness and kurtosis, respectively. \end{theorem} \begin{proof}Without loss of generality, only $rkurt$ is considered, while the logic for $|rskew|$ is the same. Additionally, the second central moments of the underlying sample distribution and consistent distribution are assumed to be 1, with other cases simply multiplying a constant factor according to Theorem 3 in REDS II. From the definition of $D$, $\frac{\varkappa\left(rkurt_{D}\right)}{rkurt_{D}}=\frac{\frac{fm_{D}-\mathrm{SWHL} fm_{D}}{\mathrm{SWHL}fm_{D}-mfm_{D}}\left(\mathrm{SWHL}fm-mfm\right)+\mathrm{SWHL} fm}{rkurt_{D}\left(\frac{var_{D}-\mathrm{SWHL} var_{D}}{\mathrm{SWHL}var_{D}-mvar_{D}}\left(\mathrm{SWHL}var-mvar\right)+\mathrm{SWHL} var\right)^2}$, where the subscript $D$ indicates that the estimates are from the central moment kernel distributions generated from the consistent distribution, while other estimates are from the underlying distribution of the sample.

Then, assuming the $m\mathbf{k}m$s are all equal to zero and $var_{D}=1$, $\frac{\varkappa\left(rkurt_D\right)}{rkurt_D}=\frac{\frac{fm_{D}-\mathrm{SWHL} fm_D}{\mathrm{SWHL}fm_D}\left(\mathrm{SWHL}fm\right)+\mathrm{SWHL} fm}{rkurt_D\left(\frac{\mathrm{SWHL}var}{\mathrm{SWHL}var_D}\right)^2}=\frac{\left(\frac{fm_{D}-\mathrm{SWHL} fm_D}{\mathrm{SWHL}fm_D}+1\right)\left(\mathrm{SWHL}fm\right)}{fm_{D}\left(\frac{\mathrm{SWHL}var}{\mathrm{SWHL}var_D}\right)^2}=\frac{\mathrm{SWHL}fm\mathrm{SWHL}var_D^2}{\mathrm{SWHL}fm_D\mathrm{SWHL} var^2}=\frac{\frac{\mathrm{SWHL} fm}{\mathrm{SWHL}var^2}}{\frac{\mathrm{SWHL}fm_D}{{\mathrm{SWHL}var_D}^2}}=\frac{\mathrm{SWHL}kurt}{\mathrm{SWHL}{kurt}_D}$. Since $\text{SWHL}fm_{D}$ are from the same fourth central moment kernel distribution as $fm_{D}=rkurt_{D}{var_{D}}^2$, according to the definition of $U$-congruence, an increase in $fm_{D}$ will also result in an increase in $\text{SWHL}fm_{D}$. Combining with Theorem 3 in REDS II, $\text{SWHL}kurt$ is a measure of kurtosis that is invariant to location and scale, so $\lim_{rkurt_{D}\rightarrow\infty}{\frac{\varkappa\left(rkurt_{D}\right)}{rkurt_{D}}}<1$. As a result, if there is at least one fixed point, let the largest one be $fix_{max}$, then it is attracting since $|\frac{\partial(\varkappa(rkurt_{D}))}{\partial(rkurt_{D})}|<1$ for all $rkurt_{D}\in [fix_{max},kurtosis_{max}]$, where $kurtosis_{max}$ is the maximum kurtosis available in $D$.

\end{proof}

As a result of Theorem \ref{twoc}, assuming continuity, $m\mathbf{k}m$s are all equal to zero, and $U$-congruence of the central moment kernel distributions, Algorithm \ref{alg:algorithm1} converges surely provided that a fixed point exists within the domain of $D$. At this stage, $D$ can only be approximated through a Monte Carlo study. The continuity of $D$ can be ensured by using linear interpolation. One common encountered problem is that the domain of $D$ depends on both the consistent distribution and the Monte Carlo study, so the iteration may halt at the boundary if the fixed point is not within the domain. However, by setting a proper maximum number of iterations, the algorithm can return the optimal boundary value. For quantile moments, the logic is similar, if the percentiles do not exceed the breakdown point. If this is the case, consistent estimation is impossible, and the algorithm will stop due to the maximum number of iterations. The fixed point iteration is, in principle, similar to the iterative reweighing in Huber $M$-estimator, but an advantage of this algorithm is that it is solely related to the inputs in Algorithm \ref{alg:algorithm1} and is independent of the sample size. Since they are consistent for a shape-scale distribution, $|rskew|$ can specify $d_{rm}$ and $d_{tm}$, $rkurt$ can specify $d_{rvar}$ and $d_{rfm}$. Algorithm \ref{alg:algorithm1} enables the robust estimations of all four moments to reach a near-consistent level for common unimodal distributions (Table \ref{tab:comparison}, SI Dataset S1), just using the Weibull distribution as the consistent distribution.

\subsection{Root mean square error }\label{2}
The SSEs of all robust estimators proposed here are often, although many exceptions exist, between those of the sample median and those of the sample mean or median central moments and $U$-central moments (SI Dataset S2). This is because similar monotonic relations between breakdown point and variance are also very common, e.g., Bickel and Lehmann \cite{bickel2012descriptive2} proved that a lower bound for the efficiency of $\text{TM}_{\epsilon}$ to sample mean is $(1-2\epsilon)^2$ and this monotonic bound holds true for any distribution. However, the direction of monotonicity differs for distributions with different kurtosis. Lehmann and Scheffé (1950, 1955) \cite{lehmann2011completeness,lehmann2012completeness} in their two early papers provided a way to construct a uniformly minimum-variance unbiased estimator (UMVUE). From that, the sample mean and unbiased sample second moment can be proven as the UMVUEs for the population mean and population second moment for the Gaussian distribution. While their performance for sub-Gaussian distributions is generally satisfied, they perform poorly when the distribution has a heavy tail and completely fail for distributions with infinite second moments. For sub-Gaussian distributions, the variance of a robust location estimator is generally monotonic increasing as its robustness increases, but for heavy-tailed distributions, the relation is reversed. So, unlike bias, the variance-optimal choice can be very different for distributions with different kurtosis.

In 1983, Lai, Robbins, and Yu proposed an estimator that adaptively chooses the mean or median in a symmetric distribution and showed that the choice is typically as good as the better of the sample mean and median regarding variance \cite{lai1983adaptive}. Another approach which can be dated back to Laplace (1812) \cite{laplace_1812} is using $w\bar{x}+(1-w)m_n$ as a location estimator and $w$ is deduced to achieve optimal variance. Inspired by Lai et al's approach \cite{lai1983adaptive}, in this study, for $rkurt$, there are 364 combinations based on 14 SWHL$fm$s and 26 SWHL$var$s (SI Text). Each combination has a root mean square error (RMSE) for a single-parameter distribution, which can be inferred using a Monte Carlo study. For $qkurt$, there are another 364 combinations, but if the percentiles of quantile moments exceed the breakdown point, that combination is excluded. Then, the combination with the smallest RMSE, calibrated by a two-parameter distribution, is chosen. Similar to Subsection \ref{sectionE}, let $I(kurtosis,dtype,n)=ikurt_{\text{WHL}fm,\text{WHL}var}$ represent these relationships. In this article, the breakdown points of the SWHLMs in $\text{SWHL}\mathbf{k}m$ were adjusted to ensure the overall breakdown points were $\frac{1}{24}$, as detailed in Theorem \ref{bdp}). There are two approaches to determine $ikurt$. The first one is computing all 364+364 $rkurt$ and $qkurt$, and then, since $\lim_{ikurt\to\infty}{\frac{I(ikurt)}{ikurt}}<1$, the same fix point iteration algorithm as Algorithm \ref{alg:algorithm1} can be used to choose the RMSE-optimum combination. The only difference is that unlike $D$, $I$ is defined to be discontinuous but linear interpolation can also ensure continuity. The second approach is shown in SI Algorithm 2. The RMSEs of these $ikurt$ from the two approaches can be further determined by a Monte Carlo study. Algorithm \ref{alg:algorithm1} can also be used to determine the optimum choice among the two approaches. The 364+364 $rkurt$ and $qkurt$ can form a vector, $Vkurt$, where the $Q_{Vkurt}(\frac{1}{5})$ to $Q_{Vkurt}(\frac{4}{5})$ can be used to determine the $d$ values of $r\mathbf{k}m$s and $q\mathbf{k}m$s. The RMSEs of those $r\mathbf{k}m$s and $q\mathbf{k}m$s can also be estimated by a Monte Carlo study and the estimator with the smallest RMSE of each ordinal is named as $i\mathbf{k}m$. When $\mathbf{k}$ is even, the $ikurt$ determined by $Ism$ (detailed in the SI Text) is used to determine $i\mathbf{k}m$. This approach yields results that are often nearly optimal (SI Dateset S1). The estimations of skewness and $i\mathbf{k}m$, when $\mathbf{k}$ is odd, follow the same logic. 

Due to combinatorial explosion, the bootstrap \cite{efron1979bootstrap}, introduced by Efron in 1979, is indispensable for computing central moments in practice. In 1981, Bickel and Freedman \cite{bickel1981some} showed that the bootstrap is asymptotically valid to approximate the original distribution in a wide range of situations, including $U$-statistics. The limit laws of bootstrapped trimmed $U$-statistics were proven by Helmers, Janssen, and Veraverbeke (1990) \cite{helmers1990bootstrapping}. In REDS I, the advantages of quasi-bootstrap were discussed \cite{richtmyer1958non,sobol1967distribution,Do1991QuasirandomRF}. By using quasi-sampling, the impact of the number of repetitions of the bootstrap, or bootstrap size, on variance is very small (SI Dataset S3). An estimator based on the quasi-bootstrap approach can be seen as a complex deterministic estimator that is not only computationally efficient but also statistical efficient. The only drawback of quasi-bootstrap compared to non-bootstrap is that a small bootstrap size can produce additional finite sample bias (SI Text). In general, the variances of invariant central moments are much smaller than those of corresponding unbiased sample central moments (deduced by Cramér \cite{cramer1999mathematical,gerlovina2019computer}), except that of the corresponding second central moment (Table \ref{tab:comparison}).

\begin{table*}
\centering
\caption{Evaluation of invariant moments for five common unimodal distributions in comparison with current popular methods}
\begin{tabular}{|l|l|l|l|l|l|l|l|l|l|l|l|l|l|l|}
\hline
Errors & $\Bar{x}$ & $\text{TM}$ & H-L&SM & $\text{HM}$& WM &SQM& BM& MoM &  MoRM & $m$HLM  & $rm_{exp,\text{BM}}$  & $qm_{exp,\text{BM}}$  \\ \hline
\midrule

$\text{WASAB}$ & 0.000 & 0.107 & 0.088 & 0.078 & 0.078 & 0.066 & 0.048 & 0.048 & 0.034 & 0.035 & 0.034 & 0.002 & 0.003 \\ \hline
$\text{WRMSE}$ & 0.014 & 0.111 & 0.092 & 0.083 & 0.083 & 0.070 & 0.053 & 0.053 & 0.041 & 0.041 & 0.038 & 0.017 & 0.018 \\ \hline
$\text{WASB}_{n=5184}$ & 0.000 & 0.108 & 0.089 & 0.078 & 0.079 & 0.066 & 0.048 & 0.048 & 0.034 & 0.036 & 0.033 & 0.002 & 0.003 \\ \hline
$\text{WSE} \lor \text{WSSE}$ & 0.014 & 0.014 & 0.014 & 0.015 & 0.014 & 0.014 & 0.014 & 0.015 & 0.017 & 0.014 & 0.014 & 0.017 & 0.017 \\ \hline

\bottomrule
\end{tabular}
\begin{tabular}{|l|l|l|l|l|l|l|l|l|l|l|l|l|l|}
\hline
Errors &  HFM$_{\mu}$ & MP$_{\mu}$& $rm$ & $qm$  & $im$  &  $var$  & $var_{bs}$& T$sd^2$ & HFM$_{\mu_2}$ & MP$_{\mu_2}$ & $rvar$ & $qvar$  & $ivar$   \\ \hline
\midrule
$\text{WASAB}$ & 0.037 & 0.043 & 0.001 & 0.002 & 0.001 & 0.000 & 0.000 & 0.200 & 0.027 & 0.042 & 0.005 & 0.018 & 0.003 \\ \hline
$\text{WRMSE}$  & 0.049 & 0.055 & 0.015 & 0.015 & 0.014 & 0.017 & 0.017 & 0.198 & 0.042 & 0.062 & 0.019 & 0.026 & 0.019 \\ \hline
$\text{WASB}_{n=5184}$  & 0.038 & 0.043 & 0.001 & 0.002 & 0.001 & 0.000 & 0.001 & 0.198 & 0.027 & 0.043 & 0.005 & 0.018 & 0.003 \\ \hline
$\text{WSE} \lor \text{WSSE}$  & 0.018 & 0.021 & 0.015 & 0.015 & 0.014 & 0.017 & 0.017 & 0.015 & 0.024 & 0.032 & 0.018 & 0.017 & 0.018\\ \hline
\bottomrule
\end{tabular}

\begin{tabular}{|l|l|l|l|l|l|l|l|l|l|l|l|l|l|l|}
\hline
Errors &  $tm$ & $tm_{bs}$ &HFM$_{\mu_3}$ & MP$_{\mu_3}$ & $rtm$ & $qtm$  & $itm$ & $fm$ &$fm_{bs}$ & HFM$_{\mu_4}$ & MP$_{\mu_4}$& $rfm$& $qfm$& $ifm$ \\ \hline
\midrule

$\text{WASAB}$ & 0.000 & 0.000 & 0.052 & 0.059 & 0.006 & 0.083 & 0.034 & 0.000 & 0.000 & 0.037 & 0.046 & 0.024 & 0.038 & 0.011 \\ \hline
$\text{WRMSE}$ & 0.019 & 0.018 & 0.063 & 0.074 & 0.018 & 0.083 & 0.044 & 0.026 & 0.023 & 0.049 & 0.062 & 0.037 & 0.043 & 0.029 \\ \hline
$\text{WASB}_{n=5184}$ & 0.001 & 0.003 & 0.052 & 0.059 & 0.007 & 0.082 & 0.038 & 0.001 & 0.009 & 0.037 & 0.047 & 0.024 & 0.036 & 0.013 \\ \hline
$\text{WSE} \lor \text{WSSE}$ & 0.019 & 0.018 & 0.021 & 0.091 & 0.015 & 0.012 & 0.017 & 0.024 & 0.021 & 0.020 & 0.027 & 0.021 & 0.020 & 0.022\\ \hline
\bottomrule
\end{tabular}
\label{tab:comparison}
\begin{minipage}{1\linewidth}
\footnotesize The first table presents the use of the exponential distribution as the consistent distribution for five common unimodal distributions: Weibull, gamma, Pareto, lognormal, and generalized Gaussian distributions. Popular robust mean estimators discussed in REDS 1 were used as comparisons. The breakdown points of mean estimators in the first table, besides H-L estimator and Huber $M$-estimator, are all ${\frac{1}{8}}$. The second and third tables present the use of the Weibull distribution as the consistent distribution not plus/plus using the lognormal distribution for the odd ordinal moments optimization and the generalized Gaussian distribution for the even ordinal moments optimization. SQM is the robust mean estimator used in recombined/quantile moments. Unbiased sample central moments ($var$, $tm$, $fm$), $U$-central moments with quasi-bootstrap ($var_{bs}$, $tm_{bs}$, $fm_{bs}$), and other estimators were used as comparisons. The generalized Gaussian distribution was excluded for He and Fung $M$-Estimator and Marks percentile estimator, since the logarithmic function does not produce results for negative inputs. The breakdown points of estimators in the second and third table, besides $M$-estimators and percentile estimator, are all ${\frac{1}{24}}$. The tables include the average standardized asymptotic bias (ASAB, as $n\rightarrow\infty$), root mean square error (RMSE, at $n=5184$), average standardized bias (ASB, at $n=5184$) and variance (SE $\lor$ SSE, at $n=5184$) of these estimators, all reported in the units of the standard deviations of the distribution or corresponding kernel distributions. W means that the results were weighted by the number of Google Scholar search results on May 30, 2022 (including synonyms). The calibrations of $d$ values and the computations of ASAB, ASB, and SSE were described in Subsection \ref{2}, \ref{1} and SI Methods. Detailed results and related codes are available in SI Dataset S1 and \href {https://zenodo.org/records/10674322} {Zenodo}.
\end{minipage}
\end{table*}

\subsection{Robustness}\label{1} 
The measure of robustness to gross errors used in this series is the breakdown point proposed by Hampel \cite{hampel1968contributions} in 1968. In REDS I, it has shown that the median of means (MoM) is asymptotically equivalent to the median Hodge-Lehmann mean. Therefore it is also biased for any asymmetric distribution. However, the concentration bound of MoM depends on $\sqrt{\frac{1}{n}}$ \cite{devroye2016sub}, it is quite natural to deduce that it is a consistent robust estimator. The concept, sample-dependent breakdown point, is defined to avoid ambiguity. 
\begin{definition}[Sample-dependent breakdown point]\label{ufbp}The breakdown point of an estimator $\hat{\theta}$ is called sample-dependent if and only if the upper and lower asymptotic breakdown points, which are the upper and lower breakdown points when $n\rightarrow\infty$, are zero and the empirical influence function of $\hat{\theta}$ is bounded. For a full formal definition of the empirical influence function, the reader is referred to Devlin, Gnanadesikan and Kettenring (1975)'s paper \cite{devlin1975robust}.

\end{definition}%
Bear in mind that it differs from the "infinitesimal robustness" defined by Hampel, which is related to whether the asymptotic influence function is bounded \cite{hampel1971general,hampel1974influence,rousseeuw2011robust}. The proof of the consistency of MoM assumes that it is an estimator with a sample-dependent breakdown point since its breakdown point is $\frac{b}{2n}$, where $b$ is the number of blocks, then $\lim_{n\rightarrow\infty}{\left(\frac{b}{2n}\right)}=0$, if $b$ is a constant and any changes in any one of the points of the sample cannot break down this estimator. 

For the $LU$-statistics, the asymptotic upper breakdown points are suggested by the following theorem, which extends the method in Donoho and Huber (1983)'s proof of the breakdown point of the Hodges-Lehmann estimator \cite{donoho1983notion}. The proof is given in the SI Text.\begin{theorem}\label{bdp}Given a U-statistic associated with a symmetric kernel of degree $\mathbf{k}$. Then, assuming that as $n\rightarrow\infty$, $\mathbf{k}$ is a constant, the upper breakdown point of the $LU$-statistic is  $1-\left(1-\epsilon_\mathbf{0}\right)^\frac{1}{\mathbf{k}}$, where $\epsilon_\mathbf{0}$ is the upper breakdown point of the corresponding $LL$-statistic.\end{theorem}

\begin{remark}If $\mathbf{k}=1$, ${1-\left(1-\epsilon_\mathbf{0}\right)}^\frac{1}{\mathbf{k}}=\epsilon_\mathbf{0}$, so this formula also holds for the $LL$-statistic itself. Here, to ensure the breakdown points of all four moments are the same, $\frac{1}{24}$, since $\epsilon_\mathbf{0}={1-\left(1-\epsilon\right)}^\mathbf{k}$, the breakdown points of all $LU$-statistics for the second, third, and fourth central moment estimations are adjusted as $\epsilon_\mathbf{0}=\frac{47}{576}$, $\frac{1657}{13824}$, $\frac{51935}{331776}$, respectively.
\end{remark}

Every statistic is based on certain assumptions. For instance, the sample mean assumes that the second moment of the underlying distribution is finite. If this assumption is violated, the variance of the sample mean becomes infinitely large, even if the population mean is finite. As a result, the sample mean not only has zero robustness to gross errors, but also has zero robustness to departures. To meaningfully compare the performance of estimators under departures from assumptions, it is necessary to impose constraints on these departures. Bound analysis \cite{gauss1823theoria} is the first approach to study the robustness to departures, i.e., although all estimators can be biased under departures from the corresponding assumptions, but their standardized maximum deviations can differ substantially \cite{bieniek2016comparison,danielak2003theory,devroye2016sub,li2018worst,bernard2020range,mathieu2022concentration}. In REDS I, it is shown that another way to qualitatively compare the estimators' robustness to departures from the symmetry assumption is constructing and comparing corresponding semiparametric models. While such comparison is limited to a semiparametric model and is not universal, it is still valid for a wide range of parametric distributions. Bound analysis is a more universal approach since they can be deduced by just assuming regularity conditions \cite{bieniek2016comparison,danielak2003theory,devroye2016sub,li2018worst,mathieu2022concentration}. However, bounds are often hard to deduce for complex estimators. Also, sometimes there are discrepancies between maximum bias and average bias. Since the estimators proposed here are all consistent under certain assumptions, measuring their biases is also a convenient way of measuring the robustness to departures. Average standardized asymptotic bias is thus defined as follows.
\begin{definition}[Average standardized asymptotic bias] For a single-parameter distribution, the average standardized asymptotic bias (ASAB) is given by $\frac{\left|\hat{\theta}-\theta\right|}{\sigma}$, where $\hat{\theta}$ represents the estimation of $\theta$, and $\sigma$ denotes the standard deviation of the kernel distribution associated with the $LU$-statistic. If the estimator $\hat{\theta}$ is not classified as an $\text{RI}$-statistic, $\text{QI}$-statistic, or $LU$-statistic, the corresponding $U$-statistic, which measures the same attribute of the distribution, is utilized to determine the value of $\sigma$. For a two-parameter distribution, the first step is setting the lower bound of the kurtosis range of interest $\Tilde{\mu}_{4_{l}}$, the spacing $\delta$, and the bin count $C$. Then, the average standardized asymptotic bias is defined as \begin{align*}\text{ASAB}_{\hat{\theta}}\coloneqq\frac{1}{C}\quad\sum_{
    \mathclap{
        \substack{
            \delta + \Tilde{\mu}_{4_{l}} \le \Tilde{\mu}_{4}\le C \delta + \Tilde{\mu}_{4_{l}}\\
            \Tilde{\mu}_{4} \textup{ is a multiple of }\delta
        }
    }
} E_{\hat{\theta}|\Tilde{\mu}_{4}}\left[\frac{\left|\hat{\theta}-\theta\right|}{\sigma}\right]
\end{align*}
where $\Tilde{\mu}_{4}$ is the kurtosis specifying the two-parameter distribution, $E_{\hat{\theta}|\Tilde{\mu}_{4}}$ denotes the expected value given fixed $\Tilde{\mu}_4$.
\end{definition}

Standardization plays a crucial role in comparing the performance of estimators across different distributions. Currently, several options are available, such as using the root mean square deviation from the mode (as in Gauss \cite{gauss1823theoria}), the mean absolute deviation, or the standard deviation. The standard deviation is preferred due to its central role in standard error estimation. In Table \ref{tab:comparison}, $\delta=0.1$, $C=70$. For the Weibull, gamma, lognormal and generalized Gaussian distributions, $\Tilde{\mu}_{4_{l}}=3$ (there are two shape parameter solutions for the Weibull distribution, the lower one is used here). For the Pareto distribution, $\Tilde{\mu}_{4_{l}}=9$. To provide a more practical and straightforward illustration, all results from five distributions are further weighted by the number of Google Scholar search results. Within the range of kurtosis setting, nearly all WLs and WHL$\mathbf{k}m$s proposed here reach or at least come close to their maximum biases (SI Dataset S4). The pseudo-maximum bias is thus defined as the maximum value of the biases within the range of kurtosis setting for all five unimodal distributions. In most cases, the pseudo-maximum biases of invariant moments occur in lognormal or generalized Gaussian distributions (SI Dataset S4), since besides unimodality, the Weibull distribution differs entirely from them. Interestingly, the asymptotic biases of $\text{TM}_{\epsilon=\frac{1}{8}}$ and $\text{WM}_{\epsilon=\frac{1}{8}}$, after averaging and weighting, are 0.107$\sigma$ and 0.066$\sigma$, respectively, in line with the sharp bias bounds of $\text{TM}_{2,14:15}$ and $\text{WM}_{2,14:15}$ (a different subscript is used to indicate a sample size of 15, with the removal of the first and last order statistics), 0.173$\sigma$ and 0.126$\sigma$, for any distributions with finite moments \cite{bieniek2016comparison,danielak2003theory}.

\section*{Discussion}
Statistics, encompassing the collection, analysis, interpretation, and presentation of data, has evolved over time, with various approaches emerging to meet challenges in practice. Among these approaches, the use of probability models and measures of random variables for data analysis is often considered the core of statistics. While the early development of statistics was focused on parametric methods, there were two main approaches to point estimation. The Gauss–Markov theorem \cite{gauss1823theoria,markov1912} states the principle of minimum variance unbiased estimation which was further enriched by Neyman (1934) \cite{neyman1934two}, Rao (1945) \cite{radhakrishna1945information}, Blackwell (1947) \cite{blackwell1947conditional}, and Lehmann and Scheffé (1950, 1955) \cite{lehmann2011completeness,lehmann2012completeness}. Maximum likelihood was first introduced by Fisher in 1922 \cite{fisher1922mathematical} in a multinomial model and later generalized by Cramér (1946), Hájek (1970), and Le Cam (1972) \cite{cramer1999mathematical,lecam1970assumptions,hajek1972local}. In 1939, Wald \cite{wald1939contributions} combined these two principles and suggested the use of minimax estimates, which involve choosing an estimator that minimizes the maximum possible loss. Following Huber's seminal work \cite{huber1964robust}, $M$-statistics have dominated the field of parametric robust statistics for over half a century. Nonparametric methods, e.g., the Kolmogorov–Smirnov test, Mann-Whitney-Wilcoxon Test, and Hoeffding's independence test, emerged as popular alternatives to parametric methods in 1950s, as they do not make specific assumptions about the underlying distribution of the data. In 1963, Hodges and Lehmann proposed a class of robust location estimators based on the confidence bounds of rank tests \cite{hodges1963estimates}. In REDS I, when compared to other semiparametric mean estimators with the same breakdown point, the H-L estimator was shown to be the bias-optimal choice, which aligns Devroye, and Lerasle, Lugosi, and Oliveira's conclusion that the median of means is near-optimal in terms of concentration bounds \cite{devroye2016sub} as discussed. The formal study of semiparametric models was initiated by Stein \cite{stein1956efficient} in 1956. Bickel, in 1982, simplified the general heuristic necessary condition proposed by Stein \cite{stein1956efficient} and derived sufficient conditions for this type of problem, adaptive estimation \cite{bickel1982adaptive}. These conditions were subsequently applied to the construction of adaptive estimates \cite{bickel1982adaptive}. It has become increasingly apparent that, in robust statistics, many estimators previously called "nonparametric" are essentially semiparametric as they are partly, though not fully, characterized by some interpretable Euclidean parameters. This approach is particularly useful in situations where the data do not conform to a simple parametric distribution but still have some structure that can be exploited. In 1984, Bickel addressed the challenge of robustly estimating the parameters of a linear model while acknowledging the possibility that the model may be invalid but still within the confines of a larger model \cite{bickel1984parametric}. He showed by carefully designing the estimators, the biases can be very small. The paradigm shift here opens up the possibility that by defining a large semiparametric model and constructing estimators simultaneously for two or more very different semiparametric/parametric models within the large semiparametric model, then even for a parametric model belongs to the large semiparametric model but not to the semiparametric/parametric models used for calibration, the performance of these estimators might still be near-optimal due to the common nature shared by the models used by the estimators. Maybe it can be named as comparametrics. Closely related topics are "mixture model" and "constraint defined model," which were generalized in Bickel, Klaassen, Ritov, and Wellner's classic semiparametric textbook (1993) \cite{bickel1993efficient} and the method of sieves, introduced by Grenander in 1981 \cite{Grenander_1981}. As the building blocks of statistics, invariant moments can reduce the overall errors of statistical results across studies and thus can enhance the replicability of the whole community \cite{leek2015reproducible,national2019reproducibility}.

\matmethods{Methods of generating the Table \ref{tab:comparison} are summarized below, with details in the SI Text. The $d$ values for the invariant moments of the Weibull distribution were approximated using a Monte Carlo study, with the formulae presented in Theorem \ref{rm} and \ref{qm}. The computation of $I$ functions is summarized in Subsection \ref{2} and further explained in the SI Text. The computation of ASABs and ASBs is described in Subsection \ref{1}. The SEs and SSEs were computed by approximating the sampling distribution using 1000 pseudorandom samples for $n=5184$ and 50 pseudorandom samples for $n=2654208$. The impact of the bootstrap size, ranging from $n=2.7\times 10^{2}$ to $n=2.765\times 10^{4}$, on the variance of invariant moments and $U$-central moments was studied using the SEs and SSEs methods described above. A brute force approach was used to estimate the maximum biases of the robust estimators discussed for the five unimodal distributions. The validity of this approach is discussed in the SI Text.


}

\showmatmethods{} 

\vspace*{-3pt}
\subsection*{Data and Software Availability} Data for Table \ref{tab:comparison} are given in SI Dataset S1-S4. The SI Datasets were deposited in \href {https://zenodo.org/records/10715792} {Zenodo}. All codes have been deposited in \href {https://github.com/johon-lituobang/REDS_Invariant_Moments} {GitHub}.
\vspace*{-3pt}
\acknow{I sincerely acknowledge the insightful comments from the editor, which considerably elevating the lucidity and merit of this paper. I am also grateful to Ruodu Wang for pointing out important mistakes regarding the $\gamma$-symmetric distribution.}

\showacknow{} 

\bibliography{pnas-sample}

\end{document}